\documentclass[12pt]{amsart}
\usepackage{amsfonts,amsmath}
\usepackage{amssymb}
\usepackage{latexsym}
\usepackage{epsfig}
\usepackage{graphicx,color,graphics}
\usepackage{subcaption}
\usepackage[T1]{fontenc}
\usepackage[utf8]{inputenc}
\usepackage{combelow}

\usepackage{newunicodechar}

\newtheorem{lemma}{Lemma}[section]

\newtheorem{theorem}[lemma]{Theorem}
\newtheorem{proposition}[lemma]{Proposition}

\newtheorem{corollary}[lemma]{Corollary}

\newcommand{\N}{\ifmmode{{\Bbb N}}\else{\mbox{${\Bbb N}$}}\fi}
\newcommand{\R}{\ifmmode{{\Bbb R}}\else{\mbox{${\Bbb R}$}}\fi}

\newcommand{\halfscript}[2]{%
  \mathord{\hbox{
    \valign{%
      \vfil##\vfil\cr
      \hbox{$#1$}\cr
      \hbox{$\scriptstyle#2$}\cr
    }%
    \kern\scriptspace
  }}%
}

\begin{document}
	
\title[Numerical study of a Transmission Problem]{Numerical study of a Transmission Problem in Elasticity with kind damping}

\author{K. Ammari}
\address{LR Analysis and Control of PDEs, LR 22ES03, Department of Mathematics, Faculty of Sciences of Monastir, University of Monastir, 5019 Monastir, Tunisia}
\email{kais.ammari@fsm.rnu.tn}

\author{V. Komornik}
\thanks{The author was supported by the following grants: NSFC No. 11871348, CAPES: No. 88881.520205/2020-01,  
MATH AMSUD: 21-MATH-03.}
\address{D\'epartement de math\'ematique, Universit\'e de Strasbourg, 7 rue René Descartes, 67084 Strasbourg Cedex, France}
\email{komornik@math.unistra.fr}

\author{M. Sep\'{u}lveda-Cortés}
\thanks{The author was supported by the following grants: Fondecyt-ANID project 1220869,
and ANID-Chile through Centro de Modelamiento  Matem\'atico  (FB210005)}
\address{DIM and CI$^2$MA, Universidad de Concepci\'on, Concepci\'on, Chile}
\email{maursepu@udec.cl}

\author{O. Vera-Villagrán}
\thanks{The author is partially financed by UTA MAYOR 2022-2023, 4764-22.}
\address{Departamento de Matem\'{a}tica, Universidad
de Tarapac\'{a}, Av. 18 de Septiembre 2222, Arica, 
Chile}
\email{opverav@academicos.uta.cl}

\date{}

\begin{abstract}
We investigate a transmission problem featuring a specific type of damping. Our primary focus is on analyzing the asymptotic behavior of the associated semigroup, $({\mathcal S}_{\mathcal A}(t))_{t\geq 0}$. We demonstrate that this semigroup exhibits a polynomial rate of decay towards zero when the initial data is taken over the domain ${\mathcal D}({\mathcal A})$. Furthermore, we establish that this decay rate is optimal. To support our theoretical findings, we present a comprehensive numerical study that validates and illustrates the sharpness of the obtained decay rates.
\end{abstract}


\maketitle
\tableofcontents

\section{Introduction}
\setcounter{equation}{0}

Localized frictional damping was studied by several authors in one or more space dimension, \cite{chen},\,\cite{ho},\,\cite{mar},\, \cite{nakao},\,\cite{zua},\,\cite{zuazua}. The main result of the above articles is that localized frictional damping produces exponential decay in time of the solution. A more general result occurs in one-dimensional space where the solution always decays exponentially to zero for any localized frictional damping active over an open subset of the domain.  See \cite{bardos} for example, where necessary and sufficient conditions are given to get stabilization of the wave equation with localized frictional damping. That is, to get the exponential stability, the damping mechanism must be present in a sufficient large neighborhood of the boundary, see also \cite{ho}.

\medskip

On the other hand, it is well-know that the viscoelastic Kelvin-Voigt damping when effective in the whole domain is stronger than the frictional damping. This damping mechanism not only produces exponential stability but also turns the corresponding semigroup into an analytic semigroup, which in particular implies that the system is exponentially stable among other important properties, see Zheng-Liu’s book \cite{liu}. But, on contrary when localized, the Kelvin-Voigt damping is weaker than the frictional damping, in the sense that the corresponding semigroup is not exponentially stable as proved in \cite{liu1}. 
In this paper we deal with the theory of elasticity. We consider the following transmission problem between two elastic materials:
\begin{center}
\includegraphics[scale=0.2]{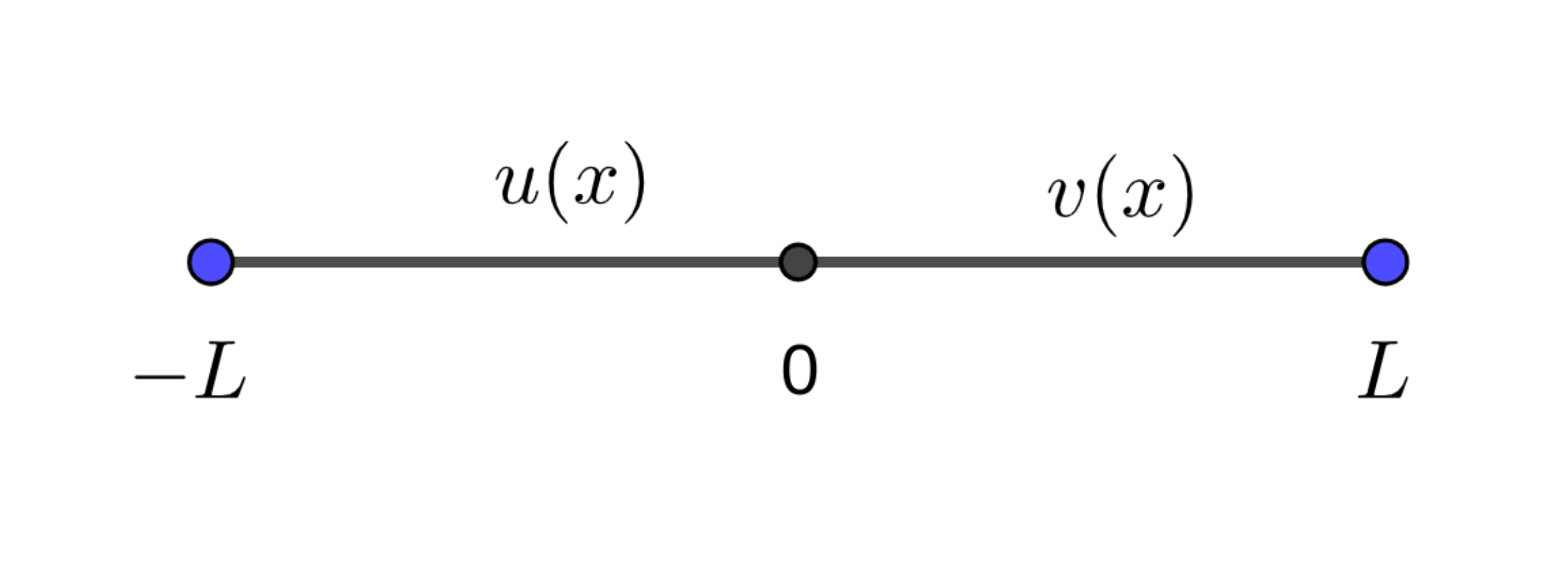} 
\end{center}

\begin{align}
\begin{cases}
\label{101}
\rho_{1}u_{tt} - k_{1}u_{xx} = 0,\quad (x,\,t)\in (-L,\,0)\times (0,\,\infty), \\
\rho_{2}v_{tt} - k_{2}v_{xx} = 0,\quad (x,\,t)\in (0,\,L)\times (0,\,\infty), 
\end{cases}
\end{align}
where $u = u(x,\,t)$ and $v = v(x,\,t)$ are real-valued functions. $k_{1}$ and $k_{2}$ are positive elastic constants. $\rho_{1},$ $\rho_{2}$ stand for the mass densities. Here we consider Dirichlet boundary conditions, which can be as
\begin{align}
\label{102}u(-L,\,t) = 0,\qquad v(L,\,t) = 0,\quad \forall\,t \geq 0.
\end{align}
The transmission conditions are given by
\begin{align}
\label{103}u(0,\,t) = v(0,\,t),\qquad k_{1}u_{x}(0,\,t) = k_{2}v_{x}(0,\,t),\quad \forall\,t\geq 0.
\end{align}
Finally, the initial data read as
\begin{align}
\label{104}
\begin{cases}
u(x,\,0) = u_{0}(x),\quad u_{t}(x,\,0) = u_{1}(x)\quad{\rm in}\quad (-L,\,0),  \\
v(x,\,0) = v_{0}(x),\quad v_{t}(x,\,0) = v_{1}(x)\quad{\rm in}\quad (0,\,L).
\end{cases}
\end{align}
In Alves {\it et. al.} the authors studied a transmission system with viscoelastic damping. In this work, we remove the viscoelastic dissipation and place dissipations with fractional derivative in the Caputo sense. Indeed, we will consider the following system
\begin{eqnarray}
\begin{cases}
\label{105}
\rho_{1}u_{tt} - k_{1}u_{xx} + \partial_{t}^{\alpha,\,\eta}u = 0,\quad (x,\,t)\in (-L,\,0)\times (0,\,\infty), \\
\rho_{2}v_{tt} - k_{2}v_{xx}  + \partial_{t}^{\alpha,\,\eta}v = 0,\quad (x,\,t)\in (0,\,L)\times (0,\,\infty), \\
u(-L,\,t) = 0,\quad v(L,\,t) = 0,\quad \forall\,t \geq 0 \\
u(0,\,t) = v(0,\,t),\quad k_{1}u_{x}(0,\,t) = k_{2}v_{x}(0,\,t),\quad \forall\,t\geq 0, \\
u(x,\,0) = u_{0}(x),\quad u_{t}(x,\,0) = u_{1}(x)\quad{\rm in}\quad (-L,\,0),  \\
v(x,\,0) = v_{0}(x),\quad v_{t}(x,\,0) = v_{1}(x)\quad{\rm in}\quad (0,\,L).
\end{cases}
\end{eqnarray}
where $u = u(x,\,t)$ and $v = v(x,\,t)$ are real-valued functions. $k_{1}$ and $k_{2}$ are positive elastic constants. $\rho_{1},$ $\rho_{2}$ stand for the mass densities. 

\medskip

The rest of the paper is divided into five sections. In Section 2, we show that the system \eqref{105} may be replaced by an augmented system \eqref{209} obtained by coupling an equation with a suitable diffusion, and we study the energy functional associated to system. In Section 3, we establish the existence and uniqueness of solutions of the system \eqref{209}. For this we use \cite{kais, kaisbis}. In section 4 we prove the strong stability of the system \eqref{209}. In Section 5, we study of the polynomial stability. 

\medskip

Throughout this paper, $C$ is a generic constant, not necessarily the same at each occasion (it may change from line to line) and depending on the indicated quantities.

\section{Augmented Model} \label{model}
\setcounter{equation}{0}

Initially, we provide a brief review of fractional calculus. For
the fractional integral, there are several slightly different
definitions for the fractional derivative operator. Our understood
the concept in the Caputo sense (see \cite{C1, C2, C3, KST}).

\medskip

Let $0 < \alpha < 1.$  The Caputo fractional integral of order
$\alpha$ is defined by
\begin{align}\nonumber
I^{\alpha}f(t) =
\frac{1}{\Gamma(\alpha)}\int_{0}^{t}(t - s)^{\alpha - 1}
f(s)ds,
\end{align}
where $\Gamma$ is the well-known gamma function, $f \in
L^{1}(0,\,+\infty).$ 
The Caputo fractional derivative operator of order $\alpha$ is
defined by
\begin{equation}
\nonumber 
D^{\alpha}f(t) = I^{1 - \alpha}Df(t):=
\frac{1}{\Gamma(1 - \alpha)}\int_{0}^{t}(t - s)^{-\alpha}
f'(s)ds,
\end{equation}
with $f \in W^{1,\,1}(0,\,+\infty).$ 
We note that Caputo definition of fractional derivative does possess a
very simple interpretation, that means, if the function $f(t)$
represents the strain history within a viscoelastic material whose
relaxation function is $[\Gamma(1 - \alpha)t^{\alpha}]^{-1}$ then
the material will experience at  any time $t$ a total stress given
the expression $D^{\alpha}f(t).$
Moreover, it easy to show that  $D^{\alpha}$ is a left inverse of
$I^{\alpha},$  but in general it is not a right inverse. Indeed, we have
\begin{eqnarray*}
D^{\alpha}I^{\alpha}f = f, \qquad I^{\alpha}D^{\alpha}f(t) =  f(t)
- f(0).
\end{eqnarray*}
For more properties of fractional
calculus  see \cite{SKM}. \\
\\
In this paper, we consider different version those
\eqref{201} and \eqref{202}.  Indeed, Choi and
MacCamy \cite{CM} establish the following definition of fractional
integro-differential operators  with weight exponential. Let  $0 <
\alpha < 1\,$ and $\eta \ge 0.$ The exponential fractional integral of
order $\alpha$ is defined by
\begin{equation}
\label{201} 
I^{\alpha,\,\eta}f(t) =
\frac{1}{\Gamma(\alpha)}\int_{0}^{t}e^{-\eta(t - s)}(t -
s)^{\alpha - 1}f(s)ds,\quad {\rm with}\  f \in L^{1}[0,\,+\infty).
\end{equation}
The exponential fractional derivative operator of order $\alpha$ is
defined by
\begin{equation}
\label{202}
\partial_{t}^{\alpha,\,\eta}f(t) =  \frac{1}{\Gamma(1 - \alpha)}
\int_{0}^{t}e^{-\eta(t - s)}(t -
s)^{-\alpha}f'(s)ds,\quad {\rm with}\  f \in W^{1,\,1}[0,\,+\infty).
\end{equation}
Note that $\partial_{t}^{\alpha,\,\eta}f(t) = I^{1 - \alpha,\,\eta} f'(t).$
\noindent
The following results are going to be used from now:
\begin{theorem}
\cite{15}
\label{theorem51} Let $\mu$ be the function
\begin{eqnarray}
\label{203}\mu(\xi) = |\xi|^{(2\alpha - 1)/2},\quad
\xi\in\mathbb{R},\quad 0 < \alpha < 1.
\end{eqnarray}
Then the relation between the {\it Input} $U$ and the {\it Output}
${\it O}$ is given by the following system
\begin{align}
\label{204}& \varphi_{t}({\pmb{\color{black}{x}}},\,t,\,\xi) + |\xi|^{2}\varphi({\pmb{\color{black}{x}}},\,t,\,\xi) =
\mu(\xi)U(t),\quad \xi\in\mathbb{R},\quad t > 0, \\
\label{205}& \varphi(0,\,\xi) = 0, \\
\label{206}& {\mathcal O} =
\pi^{-1}\sin(\alpha\pi)\int_{\mathbb{R}}\mu(\xi)\varphi(x,\,t,\,\xi)
d\xi,
\end{align}
which implies that
\begin{eqnarray}
\label{207}{\mathcal O} = I^{1 - \alpha}U,\quad {\rm where}\ U\in C([0,\,+\infty)).
\end{eqnarray}
\end{theorem}
\noindent
The strategy for to get our target is related to the elimination of the fractional derivatives in time from the domain condition in system \eqref{104}. To this, setting $\mu(\xi) = |\xi|^{(2\alpha - 1)/2},$ $\xi\in\mathbb{R},$ $\mathfrak{C} = \pi^{-1}\sin(\alpha\pi),$ and exploiting the technique from \cite{huang}, we reduce \eqref{104} to the system. In fact, we will introduce the equation
\begin{align*}
& \varphi_{t}(x,\,t,\,\xi) + (\xi^{2} + \eta)\varphi({\pmb{\color{black}{x}}},\,t,\,\xi) - u(x,\,t)\mu(\xi) = 0,\quad \xi\in\mathbb{R},\  \eta\geq 0,\  t > 0, \\
& \varphi(x,\,0,\,\xi) = 0,
\end{align*}
where $\mu(\xi) = |\xi|^{(2\alpha - 1)/2}.$ Multiplying the above equation by $e^{(\xi^{2} + \eta)t}$ and integrating we get
\begin{eqnarray*}
\varphi(t,\,\xi) = \int_{0}^{t}\mu(\xi)u(x,\,t)e^{-(\xi^{2} + \eta)(t - s)}ds
\end{eqnarray*}
and
\begin{eqnarray*}
\int_{\mathbb{R}}\mu(\xi)\varphi(x,\,t,\,\xi)d\xi = \int_{\mathbb{R}}\int_{0}^{t}\mu^{2}(\xi)u(x,\,t)e^{-(\xi^{2} + \eta)(t - s)}dsd\xi.
\end{eqnarray*}
On the other hand, using the Fubini Theorem and recalling the definition of the Gamma function, we get that
\begin{equation}
\label{208}
\partial_{t}^{\alpha,\,\eta}u(x,\,t) =  \mathfrak{C}\int_{\mathbb{R}}\mu(\xi)\varphi(x,\,t,\,\xi)d\xi.
\end{equation}
For more details of this deduction see \cite{15}. Thus, we reformulate system \eqref{103} using Theorem
\ref{theorem51}, that means, this system can be included into the augmented model
\begin{eqnarray}
\left\lbrace
\label{209}
\begin{array}{l}
\displaystyle\rho_{1}u_{tt} - k_{1}u_{xx} + \mathfrak{C}\int_{\mathbb{R}}\mu(\xi)\varphi_{1}(x,\,t,\,\xi)
d\xi = 0,  \\
\displaystyle\rho_{2}v_{tt} - k_{2}v_{xx} + \mathfrak{C}\int_{\mathbb{R}}\mu(\xi)\varphi_{2}(x,\,t,\,\xi)
d\xi  = 0 \\
\varphi_{1t}(x,\,t,\,\xi) + (|\xi|^{2} + \eta)\varphi_{1}(x,\,t,\,\xi)  = \mu(\xi)u_{t}(x,\,t), \\ 
\varphi_{2t}(x,\,t,\,\xi) + (|\xi|^{2} + \eta)\varphi_{2}(x,\,t,\,\xi)  = \mu(\xi)v_{t}(x,\,t), \\ 
u(-L,\,t) = 0,\quad v(L,\,t) = 0,\quad \forall\,t \geq 0,  \\
k_{1}u_{x}(0,\,t) = k_{2}v_{x}(0,\,t), \quad u(0,\,t) = v(0,\,t)\\
u(x,\,0) = u_{0}(x),\quad u_{t}(x,\,0) = u_{1}(x)\quad{\rm in}\quad (-L,\,0), \\
v(x,\,0) = v_{0}(x),\quad v_{t}(x,\,0) = v_{1}(x)\quad{\rm in}\quad (0,\,L).  \\
\varphi_{1}(x,\,0,\,\xi) = 0,\quad \varphi_{2}(x,\,0,\,\xi) = 0
\end{array}
\right. 
\end{eqnarray}
where $u = u(x,\,t),$ $v = v(x,\,t)$ are real-valued functions and $(x,\,t,\,\xi)\in (-L,\,L)\times (0,\,+\infty)\times \mathbb{R}.$

On the other hand, we shall consider the following technical lemma. Lemma \ref{L2.4} will be used for the proof of strong stability.
\begin{lemma}
\label{L2.4}
Let $0<\alpha<1.$ If $\eta>0$ and $\lambda\in \mathbb{R},$ or if $\eta = 0$ and $\lambda>0,$ then 
\begin{align*}	
E(\lambda,\,\alpha,\,\eta):= \displaystyle\int_{\mathbb{R}}\dfrac{|\xi|^{2\alpha - 1}d\xi}{(|\xi|^{2} + \eta + i\lambda)} < +\infty. 
\end{align*}
Furthermore, for $h\in L^{2}(\mathbb{R}; L^{2}(\Omega)),$ we have that 
\begin{align*}
H(x,\,\lambda,\,\alpha,\,\eta):= \displaystyle\int_{\mathbb{R}}\dfrac{|\xi|^{\frac{2\alpha - 1}{2}}h(x,\,\xi)d\xi}{|\xi|^{2} + \eta + i\lambda}\in L^{2}(\Omega). 
\end{align*}
\end{lemma}
\begin{proof}
Note that $E(\lambda,\,\alpha,\,\eta) = F(\lambda,\,\alpha,\,\eta) + i\lambda G(\lambda,\,\alpha,\,\eta),$ where 
\begin{align*}
F(\lambda,\,\alpha,\,\eta):=\displaystyle\int_{\mathbb{R}}\dfrac{(|\xi|^{2} + \eta)|\xi|^{2\alpha - 1}d\xi}{\lambda^{2} + (|\xi|^{2} + \eta)^{2}}\quad {\rm and}\quad G(\lambda,\,\alpha,\,\eta):=\displaystyle\int_{\mathbb{R}}\dfrac{|\xi|^{2\alpha - 1}d\xi}{\lambda^{2} + (|\xi|^{2} + \eta)^{2}}.
\end{align*}
Using that
\begin{equation*}
G(\lambda,\,\alpha,\,\eta) = 2\int_{0}^{1}\dfrac{|\xi|^{2\alpha - 1}d\xi}{\lambda^{2} + (|\xi|^{2} + \delta)^{2}} + 2\int_{1}^{+\infty}\dfrac{|\xi|^{2\alpha - 1}d\xi}{\lambda^2 + (|\xi|^{2} + \alpha)^{2}}.
\end{equation*}
Since in both cases, ($\eta > 0$ and $\lambda\in \mathbb{R}$) or ($\eta=0$ and $\lambda>0$), we obtain
$$
\dfrac{|\xi|^{2\alpha - 1}}{\lambda^2 + (|\xi|^{2} + \eta)^2}\sim \dfrac{|\xi|^{2\alpha-1}}{\lambda^2 + \eta^2}\quad \mbox{as}\quad |\xi|\to 0\quad \mbox{and}\quad \dfrac{|\xi|^{2\alpha - 1}}{\lambda^{2} + (|\xi|^{2} + \eta)^{2}}\sim \dfrac{1}{|\xi|^{5 - 2\omega}}\quad \mbox{as}\quad |\xi|\to +\infty,
$$
it follows that $G(\lambda,\,\eta) < + \infty.$ In a similar, 
\begin{equation*}
F(\lambda,\,\alpha,\,\eta) = 2\int_{0}^{1}\dfrac{(|\xi|^{2} + \eta)|\xi|^{2\alpha - 1}d\xi}{\lambda^{2} + (|\xi|^{2} + \eta)^{2}} + 2\int_{1}^{+\infty}\dfrac{(|\xi|^{2} + \eta)|y|^{2\alpha - 1}d\xi}{\lambda^2 + (|\xi|^{2} + \alpha)^{2}},
\end{equation*}		
and, if ($\eta>0$ and $\lambda\in \mathbb{R}$) or  ($\eta = 0$ and $\lambda>0$), we obtain 
$$
\dfrac{(|\xi|^{2} + \eta)|\xi|^{2\alpha - 1}}{\lambda^{2} + (|\xi|^{2} + \eta)^{2}}\sim \dfrac{(|\xi|^{2} + \eta)|\xi|^{2\alpha-1}}{\lambda^{2} + \eta^{2}}\quad \mbox{for}\quad |\xi|\to 0 
$$
and 
$$
\dfrac{(|\xi|^{2} + \eta)|\xi|^{2\alpha - 1}}{\lambda^{2} + (|\xi|^{2} + \eta)^{2}}\sim \dfrac{1}{|\xi|^{3 - 2\alpha}}\quad \mbox{for}\quad |\xi|\to +\infty.
$$	
Thus, $F(\lambda,\,\alpha,\,\eta) < +\infty,$ and consequently, it follows that $E(\lambda,\,\alpha,\,\eta) < +\infty.$ Moreover, from the Cauchy-Schwarz inequality and the fact that $h\in L^{2}(\mathbb{R}; L^{2}(0,\,L)),$ it follows that	
\begin{equation*}
\int_{\Omega}\left|H(x,\,\lambda,\,\alpha,\,\eta)\right|^{2}dx = \left(\int_{\mathbb{R}}\dfrac{|\xi|^{2\alpha - 1}d\xi}{\lambda^{2} + (|\xi|^{2} + \eta)^{2}}\right)\int_{\Omega}\int_{\mathbb{R}}|h(x,\,\xi)|^{2}d\xi dx < +\infty. 
\end{equation*}

\end{proof}

\section{Setting of the Semigroup}\label{sec:semigroup}
\setcounter{equation}{0}

In this section we establish the well-posedness of system
\eqref{209}. We thus define the phase space associated with our set of equations \eqref{209} by
\begin{align*}
\mathbb{H}^m = H^{m}(-L,\,0)\times H^{m}(0,\,L)\quad m = 1,\,2,\quad \mathbb{L}^{2} = L^{2}(-L,\,0)\times L^{2}(0,\,L).
\end{align*}
\begin{align*}
\mathbb{H}_{0}^{1} = \{(u,\,v)\in \mathbb{H}^{1};\ u(-L) = v(L) = 0,\ u(0) = v(0)\}.
\end{align*}
Under the above conditions, we have that the phase space is given by
\begin{eqnarray*}
\mathcal{H}=\mathbb{H}_{0}^{1} \times \mathbb{L}^{2}\times L^{2}(\mathbb{R};\,\mathbb{L}^{2}),
\end{eqnarray*}
which is a Hilbert space endowed with the following norm. For all $\mathbb{U}=(u,\,v,\,U,\,V,\,\varphi_{1},\,\varphi_{2})^{T}\in {\mathcal H}$ ($T$ meaning transpose),
\begin{eqnarray*}
\|\mathbb{U}\|_{{\mathcal H}}^{2}
& = &  \left[\rho_{1}\|u_{t}\|_{L^{2}(-L,\,0)}^{2} + k_{1}\|u_{x}\|_{L^{2}(-L,\,0)}^{2} + \rho_{2}\|v_{t}\|_{L^{2}(0,\,L)}^{2} + k_{2}\|v_{x}\|_{L^{2}(0,\,L)}^{2} \right. \nonumber  \\
\label{210}&  & \left.\quad +\ \mathfrak{C}\|\varphi_{1}(t)\|_{L^{2}(\mathbb{R}; L^{2}(-L,\,0))}^{2} + \mathfrak{C}\|\varphi_{2}(t)\|_{L^{2}(\mathbb{R}; L^{2}(0,\,L))}^{2} \right]
\end{eqnarray*}
and it derives from a natural inner product  $\langle \cdot , \cdot \rangle_{\mathcal{H}}$ on $\mathcal{H}$ given by
\begin{align}
\langle \mathbb{U},\,\widetilde{\mathbb{U}}\rangle_{{\mathcal H}} = &\ \rho_{1}\langle U,\,\widetilde{U} \rangle_{L^{2}(-L,\,0)} + \rho_{2}\langle V,\,\widetilde{V} \rangle_{L^{2}(0,\,L)} + k_{1}\langle u_{x},\,\widetilde{u}_{x}\rangle_{L^{2}(-L,\,0)} + k_{2}\langle v_{x},\,\widetilde{v}_{x}\rangle_{L^{2}(0,\,L)}  \nonumber \\
\label{301}& + \mathfrak{C}
\langle\varphi_{1},\,\widetilde{\varphi}_{1}\rangle_{L^{2}\left(\mathbb{R}; L^{2}(-L,\,0)\right)} + \mathfrak{C}
\langle\varphi_{2},\,\widetilde{\varphi}_{2}\rangle_{L^{2}\left(\mathbb{R}; L^{2}(0,\,L)\right)}.
\end{align}
Hence, defining $\mathbb{U}=(u,\,v,\,U,\,V,\,\varphi_{1},\,\varphi_{2})^{T},$ the set of equations \eqref{209} can be written under the form of an abstract evolution problem:
 \begin{equation}
 \label{302} \frac{d}{dt}\mathbb{U}(t) \, = \, \mathcal{A}\mathbb{U}(t), \, \forall \, t>0, \, \mathbb{U}(0) = \mathbb{U}_{0},
 \end{equation}
 where $\mathbb{U}=(u,\,v,\,U,\,V,\,\varphi_{1},\,\varphi_{2})^{T}$ and $\mathbb{U}_0 =(u_{0},\,v_{0},\,u_{1},\,v_{1},\,0,\,0)^{T}$ and the operator $\,\mathcal{A}$ is an unbounded linear operator defined as follows $\,\mathcal{A}:\mathcal{D}(\mathcal{A})\subset
\mathcal{H}\rightarrow \mathcal{H}$ with
\begin{equation}
\label{303}\mathcal{A}\left(
\begin{array}{c}
u \\
\\
v \\
\\
U \\
\\
V \\
\\
\varphi_{1} \\
\\
\varphi_{2} 
\end{array}
\right) =\left(
\begin{array}{c} 
U \\
\\
V \\
\\
\displaystyle \frac{1}{\rho_{1}}\left[k_{1}u_{xx} - \mathfrak{C}\displaystyle\int_{\mathbb{R}}\mu(\xi)\varphi_{1}(\xi)d\xi \right] \\
\\
\displaystyle \frac{1}{\rho_{2}}\left[k_{2}v_{xx} - \mathfrak{C}\displaystyle\int_{\mathbb{R}}\mu(\xi)\varphi_{2}(\xi)d\xi \right]  \\
\\
-(|\xi|^{2} + \eta)\varphi_{1}(\xi) + \mu(\xi)U \\
\\
-(|\xi|^{2} + \eta)\varphi_{2}(\xi) + \mu(\xi)V 
\end{array}
\right)
\end{equation}
with domain 
{\begin{align*}
\mathcal{D}(\mathcal{A}) = &
\left\{\mathbb{U} = (u,\,v,\,U,\,V,\,\varphi_{1},\,\varphi_{2})^{T}\in {\mathcal H}:\ U,\,V\in \mathbb{H}_{L},\ (u,\,v)\in \mathbb{H}^{2},\ k_{1}u_{x}(0) = k_{2}v_{x}(0)\right.  \\
& \quad |\xi|\varphi_{1}\in L^{2}(\mathbb{R};\,L^{2}(-L,\,0)),\,-(|\xi|^{2} + \eta)\varphi_{1} + \mu(\xi)U \in
L^{2}\left(\mathbb{R};\,L^{2}(-L,\,0)\right), \\
& \quad |\xi|\varphi_{2}\in L^{2}(\mathbb{R};\,L^{2}(0,\,L)),\,-(|\xi|^{2} + \eta)\varphi_{2} + \mu(\xi)V \in
L^{2}\left(\mathbb{R};\,L^{2}(0,\,L)\right)\}.
\end{align*}

Let the self-adjoint and strictly positive  operator $\, A :\mathcal{D}(A)\subset
H\rightarrow H$ which is given by
\begin{equation}
\label{304b} A \left(
\begin{array}{c}
u \\
v 
\end{array}
\right) =\left(
\begin{array}{c}
- \frac{k_1}{\rho_1} \, u_{xx} 
\\
- \frac{k_2}{\rho_2} \, v_{xx}  
\end{array}
\right),
\end{equation}
with the domain
$\mathcal{D}(A) =  \left\{(u,v) \in \mathbb{H}^2, u(-L)=v(L)=0, u(0)=v(0), 
k_1 u_x (0) = k_2 v_x(0) \right\}$  
and $H = \mathbb{L}^2$.

\medskip

So, the problem (\ref{103}) can be rewritten as following, as in \cite{kais,kaisbis}:

\begin{equation}
\label{abstract}
\left\{
\begin{array}{ll}
 Z_{\bf tt}({\bf t}) + AZ({\bf t}) + BB^* \partial_{\bf t}^{\alpha,\eta}Z ({\bf t})= 0, \ {\bf t} > 0 \\
Z(0) = Z_{0}, Z_{\bf t}(0) =  Z_{1},
\end{array}
\right.
\end{equation}

where $B = B^*=I_{H}, Z({\bf t}) = (u({\bf t}),\,v({\bf t})^{T}$ and 
$Z_{0} = (u_{0},\,v_{0})^{T},\quad Z_{1} = (u_{1},\,v_{1})^{T}.$

\medskip

We define 

\begin{eqnarray}
\label{301bis}\mathcal{H}_{0} & = &
\mathbb{H}_{0}^1 \times \mathbb{H}_{0}^{1} \times \mathbb{L}^{2} \times \mathbb{L}^{2}    
\end{eqnarray}
equipped with the inner product given by
\begin{align}
\langle {\mathcal U},\,\tilde{{\mathcal U}}\rangle_{{\mathcal H}_{0}} = &
\rho_1 \, \int_{-L}^{\ell}U\overline{\tilde{U}}dx +
\rho_2 \, \int_{0}^{L}V\overline{\tilde{V}}dx  +\ k_1 \, \int_{-L}^{0}u_{x}\overline{\tilde{u}}_{x}dx
+ k_2 \, \int_{0}^{L}v_{x}
\overline{\tilde{v}}_{x}dx,
\label{302bis}
\end{align}
where ${\mathcal U}=(u,\,v,\,U,\,V)^{T}$ and $\ \tilde{{\mathcal U}}=(\tilde{u},\,\tilde{v},\,\tilde{U},\,\tilde{V})^{T}.$ 

\medskip

Then, the operator 
$\,\mathcal{A}_{0}:\mathcal{D}(\mathcal{A}_{0})\subset
\mathcal{H}_{0}\rightarrow \mathcal{H}_{0}$ given by
\begin{equation}
\label{304bis}\mathcal{A}_{0} \left(
\begin{array}{c}
u \\
\\
v\\
\\
U \\
\\
V
\end{array}
\right) =\left(
\begin{array}{c}
U \\
\\
V \\
\\
\frac{k_1}{\rho_1} \, u_{xx} - U \\
\\
\frac{k_2}{\rho_2} \, v_{xx} - V 
\end{array}
\right).
\end{equation}
with the domain
\begin{align*}
\mathcal{D}({\mathcal A}_{0}) =  \left\{\mathbb{U} = (u,\,v,\,U,\,V)^{T}\in {\mathcal H}_0:\ U,\,V\in \mathbb{H}_{0}^1,\ u,\,v\in \mathbb{H}^{2},\ k_{1}u_{x}(0) = k_{2}v_{x}(0) \right\},
\end{align*}

 generates a C$_{0}$-semigroup of contractions in ${\mathcal H}_{0}, (e^{t\mathcal{A}_0})_{t \geq 0}.$ Moreover, the following auxiliary problem:

\begin{eqnarray}
\left\lbrace
\label{606}
\begin{array}{l}
\displaystyle\rho_{1}u_{tt} - k_{1}u_{xx} + u_t = 0,  \quad{\rm in}\quad (-L,\,0) \times (0,+\infty),\\
\displaystyle\rho_{2}v_{tt} - k_{2}v_{xx} + v_t  = 0, \quad{\rm in}\quad (0,\,L) \times (0,+\infty), \\
u(-L,\,t) = 0,\quad v(L,\,t) = 0,\quad \forall\,t > 0,  \\
k_{1}u_{x}(0,\,t) = k_{2}v_{x}(0,\,t), \quad u(0,\,t) = v(0,\,t)\\
u(x,\,0) = u_{0}(x),\quad u_{t}(x,\,0) = u_{1}(x)\quad{\rm in}\quad (-L,\,0), \\
v(x,\,0) = v_{0}(x),\quad v_{t}(x,\,0) = v_{1}(x)\quad{\rm in}\quad (0,\,L),   
\end{array}
\right. 
\end{eqnarray}
admits a unique solution $(u(x,t),\,v(x,t))$ such that if $(u_{0},\,v_{0},\,u_{1},\,v_{1})\in {\mathcal D}({\mathcal A}_{0})$, then the solution $(u(x,\,{\bf t}),\,v(x,\,{\bf t})$ of \eqref{606} verifies the following regularity property:
\begin{align*}
{\mathcal U}=(u,\,v,\,u_{t},\,v_{t})\in C([0,\,+\infty),\,{\mathcal D}({\mathcal A}_{0}))\cap C^{1}([0,\,+\infty),\,{\mathcal H}_{0}),
\end{align*}
and when $(u_{0},\,v_{0},\,u_{1},\,v_{1})\in {\mathcal H}_{0}$, then 
\begin{align*}
{\mathcal U}=(u,v,u_{t},v_{t})\in C([0,\,+\infty),\,{\mathcal H}_{0}).
\end{align*}

So, according to \cite{kais,kaisbis}, we have the following proposition: 
\begin{proposition}
\label{lu}
The operator $\mathcal{A}$ is the infinitesimal generator of a
contraction semigroup $(e^{t\mathcal{A}})_{t\geq 0}.$
\end{proposition}
 
From the Proposition \ref{lu}, we deduce that the system \eqref{209} is well-posed in the energy space ${\mathcal H}$ and we have the following theorem:
\begin{align}
\label{313} \mathbb{U}=(u,\,v,\,U,\,V,\,\varphi_{1},\,\varphi_{2})^{T}\quad
{\rm and} \quad
\mathbb{U}_{0}=(u_{0},\,v_{0},\,u_{1},\,u_{1},\,0,\,0)^{T},
\end{align}
\begin{corollary}(Existence and uniqueness of solutions)
\label{exis}
If $(u_{0},\,v_{0},\,u_{1},\,v_{1},\,0,\,0)\in {\mathcal H},$ 
the problem \eqref{209} admits a unique solution
\begin{equation*}
(u,\,v,\,u_{t},\,v_{t},\,\varphi_{1},\,\varphi_{2}) \in C\left([0,\,+\infty); {\mathcal H}\right),
\end{equation*}
and for $(u_{0},\,v_{0},\,u_{1},\,v_{1},\,0,\,0)\in {\mathcal D}({\mathcal A})$, the problem \eqref{209} admits a unique solution
\begin{equation*}
(u,\,v,\,u_{t},\,v_{t},\,\varphi_{1},\,\varphi_{2})\in C\left([0,\,+\infty); {\mathcal D}({\mathcal A})\right)\cap C^{1}\left([0,\,+\infty); {\mathcal H}\right).
\end{equation*}
Moreover, the energy in time $t\geq 0$ is given by
\begin{align}
E(t) = &\ \frac{1}{2}\left[\rho_{1}\|u_{t}\|_{L^{2}(-L,\,0)}^{2}dx + k_{1}\|u_{x}\|_{L^{2}(-L,\,0)}^{2} + \rho_{2}\|v_{t}\|_{L^{2}(0,\,L)}^{2} + k_{2}\|v_{x}\|_{L^{2}(0,\,L)}^{2} \right. \nonumber  \\
\label{213}& \left.\qquad +\ \mathfrak{C}\|\varphi_{1}(t)\|_{L^{2}(\mathbb{R}; L^{2}(-L,\,0))}^{2} + \mathfrak{C}\|\varphi_{2}(t)\|_{L^{2}(\mathbb{R}; L^{2}(0,\,L))}^{2}  \right]
\end{align}
which satisfies
\begin{align}
\label{314}\frac{d}{dt}E(t) = &- \mathfrak{C}\int_{\mathbb{R}}(|\xi|^{2} + \eta)\|\varphi_{1}(t,\,\xi)\|_{L^{2}(-L,\,0)}^{2}d\xi - \mathfrak{C}\int_{\mathbb{R}}(|\xi|^{2} + \eta)\|\varphi_{2}(t,\,\xi)\|_{L^{2}(0,\,L)}^{2}d\xi. 
\end{align}
\end{corollary}

\section{Strong stability}

In this section, the following theorem plays an important role.
\begin{theorem}(Arendt-Batty \cite{Arendt})
\label{Arendt}
	Let ${\mathcal A}$ be the generator of a C$_{0}$-semigroup $({\mathcal S}(t))_{t\ge 0}$ in a reflexive Banach space $X.$ If the following conditions are satisfied:
	\begin{enumerate}
		\item [(i)] ${\mathcal A}$ has no purely imaginary eigenvalues,
		\item [(ii)] $\sigma({\mathcal A})\cap i\mathbb{R}$ is countable,
	\end{enumerate}
	then, $\{{\mathcal S}(t)\}_{t\ge 0}$ is strongly stable.
\end{theorem}
\begin{proposition}
	\label{P4.1}
	If $\lambda \in \mathbb{R},$ then $i\lambda I - {\mathcal A}$ is Injective.
\end{proposition}
\begin{proof}
Let $\lambda\in \mathbb{R}$ Such that $i\lambda$ is an eigenvalue of the operator ${\mathcal A},$ and let $\mathbb{U}=(u,\,v,\,U,\,V,\,\varphi_{1},\,\varphi_{1})\in \mathcal{D}({\mathcal A})$ be the associated eigenvector. Then ${\mathcal A}\mathbb{U} = i\lambda \mathbb{U}.$ Equivalently 	
\begin{align}
\label{26}
\begin{cases}
U = i\lambda u,  \\
V = i\lambda v, \\
k_{1}u_{xx} -  \mathfrak{C} \displaystyle\int_{\mathbb{R}}\mu(\xi)\varphi_{1}(\xi)d\xi = i\lambda\rho_{1}U,  \\
k_{2}v_{xx} - \mathfrak{C} \displaystyle\int_{\mathbb{R}}\mu(\xi)\varphi_{2}(\xi)d\xi = i\lambda\rho_{2}V \\
(|\xi|^{2} + \eta + i\lambda)\varphi_{1}(\xi) = \mu(\xi)U,  \quad \forall \,\xi\in \mathbb{R} \\
(|\xi|^{2} + \eta + i\lambda)\varphi_{2}(\xi) = \mu(\xi)V,  \quad \forall \,\xi\in \mathbb{R} 
\end{cases}
\end{align}
Note that
\begin{eqnarray*}
& &0 = Re\langle i\lambda \mathbb{U},\,\mathbb{U}\rangle_{\mathcal{H}} \\
& = & -\mathfrak{C}\int_{\mathbb{R}}(|\xi|^{2} + \eta)\|\varphi_{1}(t,\,\xi)\|_{L^{2}(-L,\,0)}^{2}d\xi -  \mathfrak{C}\int_{\mathbb{R}}(|\xi|^{2} + \eta)\|\varphi_{2}(t,\,\xi)\|_{L^{2}(0,\,L)}^{2}d\xi.
\end{eqnarray*}
Therefore
\begin{align}
\label{27}
\begin{cases}
\varphi_{1}(x,\,\xi) = 0\quad \mbox{a.\ e.}\quad {\rm in}\quad (x,\,\xi)\in(-L,\,0)\times\mathbb{R}.\\
\varphi_{2}(x,\,\xi) = 0\quad \mbox{a.\ e.}\quad {\rm in}\quad (x,\,\xi)\in(0,\,L)\times\mathbb{R}.
\end{cases}
\end{align} 
Applying \eqref{27} to \eqref{26}$_{5,\,6}$, we obtain:
\begin{equation}
\label{28}
\begin{cases}
U(x) = 0 \quad  \mbox{a.\ e.}\quad {\rm in}\quad x\in(-L,\,0) \\
V(x) = 0 \quad  \mbox{a.\ e.}\quad {\rm in}\quad x\in(0,\,L).
\end{cases}
\end{equation} 
Now, applying \eqref{28}$_{1,\,2}$ to \eqref{26}$_{1,\,2},$ we have
\begin{align}
\label{29}
\begin{cases}
i\lambda u(x) = 0\quad \mbox{a.\ e.}\quad {\rm in}\quad x\in (-L,\,0), \\
i\lambda v(x) = 0\quad \mbox{a.\ e.}\quad {\rm in}\quad x\in (0,\,L).
\end{cases}
\end{align}
If $\lambda\neq 0$, then $u = 0$ almost everywhere on $(-L,\,0).$ Thus, from \eqref{26}$_{3},$ it follows that $u_{xx}=0$ almost everywhere on $(-L,\,0).$ From the boundary conditions, it follows that $u\equiv 0,$ and therefore, the third equation implies that $U(x) = 0$ almost everywhere in $x\in(-L,\,0).$ Similarly, we obtain $V(x) = 0$ almost everywhere in $x\in(0,\,L).$ \\
Assuming that $\lambda = 0,$ from the third and fourth equations of the system \eqref{26}, along with the boundary conditions of the problem, we obtain the following system
\begin{align}
\begin{cases}
k_{1}u_{xx} = 0 \\
k_{1}v_{xx} = 0 \\
u(-L,\,t) = 0,\quad u(L,\,t) = 0,\quad \forall\,t\geq 0, \\
k_{1}u_{x}(0,\,t) = k_{2}v_{x}(0,\,t),\quad u(0,\,t) = v(0,\,t).
\end{cases}
\end{align}
Then we obtain that $u\equiv 0$ and $v\equiv 0.$ Therefore, in any case, $ker(i\lambda I - {\mathcal A}) = \{0\}.$  
\end{proof}
\begin{corollary}
\label{C4.2}
If $\lambda \in \mathbb{R},$ then $i\lambda$ is not an eigenvalue of ${\mathcal A}$
\end{corollary}
\begin{proposition}\label{p4.2}
If $\eta = 0,$ then the operator ${\mathcal A}$ is not invertible and consequently $0\in \sigma({\mathcal A}).$
\end{proposition}
\begin{proof}
Let $W_{0}=\left(\sin(\pi x/L),\,0,\,0,\,0,\,0,\,0\right)\in \mathcal{H}$ and we assume that there exists $\mathbb{U}_{0} = (u_{0},\,v_{0},\,U_{0},\,V_{0},\,\varphi_{1,\,0},\,\varphi_{2,\,0})\in \mathcal{D}({\mathcal A})$ such that ${\mathcal A}\mathbb{U}_{0} = \mathbb{F}_{0}.$ In this case, $\varphi_{0}(\xi) = |\xi|^{\frac{2\alpha - 5}{2}}\sin(\pi x/L).$ However, $\varphi_{1,\,0}\notin L^{2}(\mathbb{R}; L^{2}(-L,\,0))$ and $\varphi_{2,\,0}\notin L^{2}(\mathbb{R}; L^{2}(0,\,L))$ for $0 < \alpha < 1.$ 

\end{proof}

\begin{proposition}
	\label{P4.3}
	\begin{enumerate}
		\item [(a)] If $\eta = 0,$ then $i\lambda I - {\mathcal A}$ is surjective, for any $\lambda\neq 0.$
		\item [(b)] If $\eta > 0$ and $\lambda\in \mathbb{R},$ then $i\lambda I - {\mathcal A}$ is surjective. 
	\end{enumerate}
\end{proposition}
\begin{proof}
Given $\mathbb{F} = (f_{1},\,f_{2},\,f_{3},\,f_{4},\,f_{5},\,f_{6})^{T}\in \mathcal{H},$ we aim to show that there exists a vector $\mathbb{U}=(u,\,v,\,U,\,V,\,\varphi_{1},\,\varphi_{2})^{T}\in {\mathcal D}({\mathcal A})$ such that $(i\lambda I - {\mathcal A})\mathbb{U} = \mathbb{F}.$ That is,
\begin{align}\
	\label{31}
	\begin{cases}
		i\lambda u - U = f_{1} \iff U = i\lambda u - f_{1} \\
		i\lambda v - V = f_{2},  \iff V = i\lambda v - f_{2} \\
		i\lambda U - k_{1}u_{xx} + \mathfrak{C} \displaystyle\int_{\mathbb{R}}\mu(\xi)\varphi_{1}(\xi) d\xi = f_{3},  \\
i\lambda V - k_{2}v_{xx} + \mathfrak{C} \displaystyle\int_{\mathbb{R}}\mu(\xi)\varphi_{2}(\xi) d\xi = f_{4}, \\
		(|\xi|^{2} + \eta + i\lambda)\varphi_{1}(\xi) - \mu(\xi)U = f_{5}(\xi), \quad \forall\, \xi\in \mathbb{R},\\
		 (|\xi|^{2} + \eta + i\lambda)\varphi_{2}(\xi) - \mu(\xi)V = f_{6}(\xi), \quad \forall\, \xi\in \mathbb{R}
	\end{cases}
\end{align}
Suppose that $\lambda \neq 0$ and let $\eta\ge 0.$ Replacing \eqref{31}$_{1,\,2}$ into \eqref{31}$_{5,\,6}$ respectively we obtain 
\begin{align}
\begin{cases}
\displaystyle\varphi_{1}(\xi) = \frac{f_{5}(\xi)}{(|\xi|^{2} + \eta + i\lambda)} - \dfrac{\mu(\xi)f_{1}}{(|\xi|^{2} + \eta + i\lambda)} + \frac{i\lambda\mu(\xi)u}{(|\xi|^{2} + \eta + i\lambda)} \\
\displaystyle\varphi_{2}(\xi) = \frac{f_{6}(\xi)}{(|\xi|^{2} + \eta + i\lambda)} - \dfrac{\mu(\xi)f_{2}}{(|\xi|^{2} + \eta + i\lambda)} + \frac{i\lambda\mu(\xi)v}{(|\xi|^{2} + \eta + i\lambda)}
\end{cases}
\end{align}
From Lemma \ref{L2.4}, it follows that
\begin{align*}
\begin{cases}
\displaystyle\mathfrak{C}\int_{\mathbb{R}}\mu(\xi)\varphi_{1}(\xi)d\xi
= \mathfrak{C}\left[\int_{\mathbb{R}}\dfrac{\mu(\xi)f_{5}(\xi)d\xi}{(|\xi|^{2} + \eta + i\lambda)} + E_{1}(\lambda,\,\alpha,\,\eta)(i\lambda u - f_{1})\right] \\
\displaystyle\mathfrak{C}\int_{\mathbb{R}}\mu(\xi)\varphi_{2}(\xi)d\xi
= \mathfrak{C}\left[\int_{\mathbb{R}}\dfrac{\mu(\xi)f_{6}(\xi)d\xi}{(|\xi|^{2} + \eta + i\lambda)} + E_{2}(\lambda,\,\alpha,\,\eta)(i\lambda v - f_{2})\right]
\end{cases}
\end{align*}
Thus, from the remaining equations in the system \eqref{31}, it follows that
\begin{align*}
\begin{cases}
\displaystyle-\lambda^{2}u - i\lambda f_{1} - k_{1}u_{xx} + \mathfrak{C}E_{1}(\lambda,\,\alpha,\,\eta)(i\lambda u - f_{1}) + \mathfrak{C}\int_{\mathbb{R}}\dfrac{\mu(\xi)f_{5}(\xi)d\xi}{(|\xi|^{2} + \eta + i\lambda)} 
= f_{3} \\
\displaystyle-\lambda^{2}v - i\lambda f_{2} - k_{2}v_{xx} + \mathfrak{C}E_{2}(\lambda,\,\alpha,\,\eta)(i\lambda v - f_{2}) + \mathfrak{C}\int_{\mathbb{R}}\dfrac{\mu(\xi)f_{6}(\xi)d\xi}{(|\xi|^{2} + \eta + i\lambda)} 
= f_{4}.
\end{cases}
\end{align*}
Then
\begin{align*}
\begin{cases}
\displaystyle-\lambda^{2}u - k_{1}u_{xx} + i\lambda \mathfrak{C}E_{1}(\lambda,\,\alpha,\,\eta)u  
= i\lambda f_{1} + f_{3} + \mathfrak{C}E_{1}(\lambda,\,\alpha,\,\eta)f_{1} -  \mathfrak{C}\int_{\mathbb{R}}\dfrac{\mu(\xi)f_{5}(\xi)d\xi}{(|\xi|^{2} + \eta + i\lambda)}\\
\displaystyle-\lambda^{2}v - k_{2}v_{xx} + i\lambda\mathfrak{C}E_{2}(\lambda,\,\alpha,\,\eta)v  
= i\lambda f_{2} + f_{4} + \mathfrak{C}E_{2}(\lambda,\,\alpha,\,\eta)f_{2} - \mathfrak{C}\int_{\mathbb{R}}\dfrac{\mu(\xi)f_{6}(\xi)d\xi}{(|\xi|^{2} + \eta + i\lambda)}.
\end{cases}
\end{align*}
However
\begin{align*}
\begin{cases}
\displaystyle\left|\mathfrak{C}\int_{-L}^{0}\widetilde{U}\displaystyle\int_{\mathbb{R}}\frac{\mu(\xi)f_{5}(\xi)d\xi}{|\xi|^{2} + \eta + i\lambda}dx\right|\leq |L|\mathfrak{C}H_{1}(x,\,\lambda,\,\alpha,\,\eta)\|\widetilde{U}\|_{\mathbb{H}_{L}^{1}},\\
\displaystyle\left|\mathfrak{C}\int_{0}^{L}\widetilde{V}\displaystyle\int_{\mathbb{R}}\frac{\mu(\xi)f_{6}(\xi)d\xi}{|\xi|^{2} + \eta + i\lambda}dx\right|\leq |L|\mathfrak{C}H_{2}(x,\,\lambda,\,\alpha,\,\eta)\|\widetilde{V}\|_{\mathbb{H}_{L}^{1}}.
\end{cases}
\end{align*}
So is enough to proceed in a similar way to the approach used in the proof of Proposition 3.1 and using Lax-Milgram’s Theorem.
Finally, If $\lambda = 0$ and $\eta >0,$\ we have $U = - f_{1},$\ $\ V = - f_{2},$\ $\ \displaystyle\varphi_{1}(\xi) = \dfrac{f_{5}(\xi)}{|\xi|^{2} + \eta} - \frac{\mu(\xi)f_{1}}{|\xi|^{2} + \eta}$\ and $\ \displaystyle\varphi_{2}(\xi) = \dfrac{f_{6}(\xi)}{|\xi|^{2} + \eta} - \frac{\mu(\xi)f_{2}}{|\xi|^{2} + \eta}$\
\begin{align}
\label{3333}
\begin{cases}
-k_{1}u_{xx} + \mathfrak{C} \displaystyle\int_{\mathbb{R}}\mu(\xi)\varphi_{1}(\xi)d\xi = f_{3},  \\
-k_{2}v_{xx} + \mathfrak{C} \displaystyle\int_{\mathbb{R}}\mu(\xi)\varphi_{2}(\xi)d\xi = f_{4}.
\end{cases}
\end{align}
Since Lax-Milgram's Theorem applied to the system \eqref{3333} the result follows. \end{proof} 
\begin{corollary}
	\label{C4.4} 	
	\begin{enumerate}
		\item [(a)] If $\eta=0,$ then $i\lambda \notin \sigma({\mathcal A}),$ for any $\lambda\neq 0,$
		\item [(b)] If $\eta> 0,$ then $i\lambda \notin \sigma({\mathcal A}),$ for any $\lambda\in \mathbb{R}.$
	\end{enumerate}
\end{corollary}

\begin{theorem}
\label{T4.5}
The C$_{0}$-semigroup $\{{\mathcal S}(t)\}_{t\ge 0}$ generated by operator ${\mathcal A}$ is strongly stable in $\mathcal{H},$ i. e.,
	$$
	\lim_{t\to +\infty}\left\|e^{t{\mathcal A}}\mathbb{U}_{0}\right\|_{\mathcal{H}} = 0, \quad\forall \; \mathbb{U}_{0}\in \mathcal{H}.
	$$
\end{theorem}
\noindent
\proof From the Corollary $\ref{C4.2},$ it follows that the operator ${\mathcal A}$ does not have purely imaginary eigenvalues. However, if $\eta = 0,$ Proposition $\ref{p4.2}$ and item (a) of the Corollary $\ref{C4.4},$ imply that $\sigma({\mathcal A}) \cap i\mathbb{R} = \{0\}.$ In the case of $\eta > 0,$ using item (b) of the corollary $\ref{C4.4},$ we conclude that $\sigma({\mathcal A})\cap i\mathbb{R} = \emptyset.$ Therefore, in both cases, we can apply Arendt and Batty's Theorem, leading to the desired result. \qquad $\square$

\section{Polynomial stability for ${\pmb{\color{black}{\eta\geq 0}}}$}
\setcounter{equation}{0}

To obtain the rate of stability, for $\eta=0$, we need the following result (see \cite{batty},\ Theorem 8.4).
\begin{theorem}
\label{thm5.1}
\cite{batty}. Let ${\mathcal S}(t)$ be a bounded C$_{0}$-semigroup on a Hilbert space ${\mathcal H}$ with generator ${\mathcal A}.$ Assume that $\sigma({\mathcal A})\cap i\mathbb{R} = \{0\}$ and that there exist $\beta\geq 1,$ $\gamma>0$ such that 
\end{theorem}
\begin{align*}
\|(i\lambda I - {\mathcal A})^{-1}\|_{{\mathcal L}(X)} \leq 
\begin{cases}
{\mathcal O}(|\lambda|^{-\beta}),\ \lambda\rightarrow 0 \\
{\mathcal O}(|\lambda|^{\gamma}),\ \lambda\rightarrow +\infty.
\end{cases}
\end{align*}
Then there exist a constant $C>0$ such that
\begin{align*}
\|{\mathcal S}(t)z\|_{\mathcal H} \leq \frac{C}{t^{\frac{1}{\max\{\beta,\,\gamma\}}}}\|z\|_{{\mathcal D}({\mathcal H})},\quad \forall\,t>0,\ z\in {\mathcal D}({\mathcal A})\cap {\mathcal R}({\mathcal A}),
\end{align*}
where ${\mathcal D}({\mathcal A})$ is the domain of ${\mathcal A}$ and ${\mathcal R}({\mathcal A})$ is the range of ${\mathcal A}.$ \\
Using the Theorem 5.1, a simple adaptation of the proof of \cite{bena} (Theorems 5.8 and 5.9) (see also \cite{kais}) lead to the following stability result:
\begin{proposition}
\label{prop5.2}
If $\eta = 0,$ then there exist a $C$ such that 
\begin{align}
\|e^{t{\mathcal A}}\mathbb{U}_{0}\|_{\mathcal H} \leq \frac{C}{\sqrt{t}}\|\mathbb{U}_{0}\|_{\mathcal H},\quad \forall\,t>0,\ \mathbb{U}_{0}\in {\mathcal D}({\mathcal A})\cap R({\mathcal A}),
\end{align}
with ${\mathcal R}({\mathcal A})$ is the range of ${\mathcal A}.$
\end{proposition}
\label{prop5.3}
{\bf Polynomial stability for $\eta > 0.$} The semigroup  $(e^{t\mathcal{A}_0})_{t \geq 0}$ is exponential stable, i.e., there exist $\delta,\,C>0,$ such that:
\begin{align*}
\|e^{t{\mathcal A}_{0}}\|_{{\mathcal L}(\mathcal{H}_0)} \leq Ce^{-\delta t},\quad \forall\,t\geq 0.
\end{align*}
Applying (\cite{kaisbis},\ Corollary 2.6.1) we obtain the following polynomial decay result for the system \eqref{104} and the augmented system \eqref{209}.
\begin{proposition}
If $\eta > 0,$ then the semigroup $(e^{t{\mathcal A}})_{t\geq 0}$ is polynomialy stable, namely there exists a constant $C>0$ such that
\begin{align*}
& \|e^{t{\mathcal A}}(u_{0},\,v_{0},\,u_{1},\,v_{1},\,\varphi_{0,\,1},\,\varphi_{0,\,2})\|_{\mathcal H} \\
& \qquad \frac{C}{(1 + t)^{\frac{1}{1 - \alpha}}}\|(u_{0},\,v_{0},\,u_{1},\,v_{1},\,\varphi_{0,\,1},\,\varphi_{0,\,2})\|_{{\mathcal D}({\mathcal A})},\quad \forall\,t\geq 0,
\end{align*}
for all initial data $(u_{0},\,v_{0},\,u_{1},\,v_{1},\,\varphi_{0,\,1},\,\varphi_{0,\,2})\in {\mathcal D}({\mathcal A}).$ In particular, the energy of the strong solution of \eqref{104} and \eqref{209} satisfies the following estimate:
\begin{align*}
E(t) \leq \frac{C}{(1 + t)^{\frac{2}{1 - \alpha}}}\|(u_{0},\,v_{0},\,u_{1},\,v_{1},\,0,\,0)\|^2_{{\mathcal D}({\mathcal A})},\quad \forall \,t > 0. 
\end{align*}
\end{proposition}

\section{Numerical Approximation}

In this section we will verify numerically the polynomial rate of decay obtained in the previous 
section. 

\subsection{Finite Volume Approximation}
We consider the finite volume method (FVM) for the spatial discretization of the variables $u=u(x,t)$ and $v=v(x,t)$, based on a finite difference discretization of the flux \cite{Eymard}. In this sense, let $(-L,L)$ be a uniform discretization of the interval $(-L,L)$ in small $2J+1$ control volumes 
$K_j=(x_{j-\frac{1}{2}},x_{j+\frac{1}{2}})$,
with $x_{j+\frac{1}{2}}=j\delta x$, $\delta x = \dfrac{L}{J}$, 
$j=-J,\ldots,J$.
The unknowns $u(x,t)$ and $v(x,t)$, are approximated by $u_j(t)$ and $v_j(t)$ respectively in the control volume $K_j$. In turn, the transmission condition \eqref{103} implies $u_0(t)=v_0(t)$.
Given the uniformity of the mesh, the diffusion term in the respective domains $(-L,0)$ and $(0,L)$ are approximated simply as:
\begin{align}
    &k_1u_{xx}\approx 
    \frac{1}{\delta x}\left(k_1\frac{u_{j+1}-u_j}{\delta x} - k_1\frac{u_{j+1}-u_j}{\delta x}\right)=
    k_1\frac{u_{j-1}-2u_j+u_j}{\delta x^2},
    \qquad j=-J,\ldots,-1
    \label{diff1}
    \\
    &k_2v_{xx}\approx 
    \frac{1}{\delta x}\left(k_2\frac{v_{j+1}-v_j}{\delta x} - k_2\frac{v_{j+1}-v_j}{\delta x}\right)=
    k_2\frac{v_{j-1}-2v_j+v_j}{\delta x^2},
    \qquad j=1,\ldots,J.
    \label{diff2}
    \end{align}
 On the other hand, assuming that  $u_0$ and $v_0$ are constant in $(x_{-\frac{1}{2}},0)$ 
 and  $(0,x_{\frac{1}{2}})$ respectively, 
 then integrating the equation \eqref{105}
 in the volume $K_0$,
and taking into account the transmission condition
\eqref{103}$_2$
 we obtain:
$$
\delta x\left(\frac{\rho_1}{2}\ddot{u}_0 + \frac{\rho_2}{2}\ddot{v}_0\right)
- \left(k_2\frac{v_{1}-v_{0}}{\delta x} - k_1\frac{u_{0}-u_{-1}}{\delta x}\right)
+\partial_{\alpha,\eta} \frac{u_0+v_0}{2}
=0
$$
 Then, taking into additional consideration the fact that $u_0=v_0$, we have that
\begin{equation}
  \frac{\rho_1+\rho_2}{2}\ddot{u}_0
  -
    \frac{k_1 u_{-1}-(k_1+k_2)u_0+k_2v_1}{\delta x^2}  
+\partial_{\alpha,\eta}u_0=0.    
\label{transm}
\end{equation}
\subsection{Linear equations of Motion}
    Let the vector 
$\mathbf{U}=[\mathbf{u}(t), \mathbf{v}(t)]^\top
=[u_{-J}(t),\ldots,u_0(t)=v_0(t),\ldots, v_J(t)]^\top$ is
an 
approximation of $[u,v]^\top$
in $\mathbb{R}^{2J+1}$.
Considering \eqref{diff1}, \eqref{diff2} and \eqref{transm}, we have the following system of equations of motion
\begin{equation}
	\label{LEM}
	\mathbf{M}\ddot{\mathbf{U}}(t)+
	\mathbf{K}{\mathbf{U}}(t)+
	\mathbf{C}\overset{\alpha,\eta}{\mathbf{U}}(t)
	=0
\end{equation}
where $\mathbf{M}$ is the $2J+1$-diagonal densities matrix
given by:
$$
\mathbf{M}=
\begin{pmatrix}
\rho_1\mathbf{I}_{J\times J} & \mathbf{0}_J& \mathbf{O}_{J\times J}\\
 \mathbf{0}_J^\top & \frac{\rho_1+\rho_2}{2} &  \mathbf{0}_J^\top \\
 \mathbf{O}_{J\times J} & \mathbf{0}_J & \rho_2\mathbf{I}_{J\times J}
\end{pmatrix},
$$
with $\mathbf{I}_{J\times J}$ and $\mathbf{O}_{J\times J}$ the identity and null matrices respectively, and $\mathbf{0}_J$ the column vector of $\mathbb{R}^J$; 
The stiffness matrix $\mathbf{K}$ is given by
$$
\mathbf{K}=\frac{1}{\delta x^2}
\begin{pmatrix}
-k_1 \mathbf{D}^2 & \begin{bmatrix}
    \mathbf{0}_{J-1} \\ -k_1 
\end{bmatrix}   
&
\mathbf{O}_{J\times J}\\
\begin{bmatrix}
    \mathbf{0}_{J-1}^\top & -k_1 &
\end{bmatrix}
&
k_1+k_2 &
\begin{bmatrix}
   -k_2 & \mathbf{0}_{J-1}^\top 
\end{bmatrix}   
\\
\mathbf{O}_{J\times J} & 
\begin{bmatrix}
   -k_2 \\ \mathbf{0}_{J-1} 
\end{bmatrix}   
& 
-k_2 \mathbf{D}^2
\end{pmatrix},
\qquad
\mathbf{D}^2=
\begin{pmatrix}
    -2 & 1 & \dots & 0 \\
    1 & -2 & 1 & \dots \\
  \vdots  & \ddots & \ddots &  \ddots \\
0 & \dots & 1 & -2
\end{pmatrix}
$$
Last but not least
$\mathbf{C}\overset{\alpha,\eta}{\mathbf{U}}(t)= D^{\alpha,\eta}{\mathbf{U}}(t)$
is the generalized Caputo fractional derivative defined in \eqref{202}. 
However, in order for the computed energy to be numerically decreasing (without oscillations) in the case of dissipation with fractional derivatives, we will use the formula \eqref{208}, which in this case becomes:
\begin{equation}
\label{623}
D^{\alpha,\eta}{\mathbf{U}}(t) =  \mathfrak{C}\int_{\mathbb{R}}\mu(\xi)
\mathbf{\Phi}(t,\,\xi)d\xi.
\end{equation}
where $\mathbf{\Phi}
=[\varphi_{-J}(t,\xi),\ldots,\varphi_J(t,\xi)]^\top
\in \mathbb{R}^{2J+1}$ is
an 
approximation of the
solutions of \eqref{209}$_3$ and \eqref{209}$_4$.
That is, by theorem \ref{theorem51}, the vector 
$\mathbf{\Phi}$ is a solution of the following system
\begin{align}
\label{624}& \mathbf{\Phi}_{t}(t,\,\xi) + |\xi|^{2}\mathbf{\Phi}(t,\,\xi) =
\mu(\xi)\mathbf{U}(t),\quad \xi\in\mathbb{R},\quad t > 0, \\
\label{625}& \mathbf{\Phi}(0,\,\xi) = 0.
\end{align}

\subsection{Time discretization}

In order to preserve the energy with a second order scheme in time, we choose a $\beta$-Newmark scheme for $w$.
The method  consists of updating
the displacement, velocity and acceleration vectors
 at the current time $t^n=n\delta t$ to the time $t^{n+1} = (n+1)\delta t$, a small time interval
 $\delta t$
later.
The Newmark algorithm \cite{Newmark}
is based on a set of two relations expressing the forward displacement
$\mathbf{U}^{n+1}$
 and velocity
$\dot{\mathbf{U}}^{n+1}$  in terms of their
current values and the forward and current values of the acceleration:
\begin{eqnarray}
\dot{\mathbf{U}}^{n+1} &=& \dot{\mathbf{U}}^{n} + (1-\widetilde{\gamma})\delta t\,\ddot{\mathbf{U}}^{n} + \widetilde{\gamma}\delta t\,\ddot{\mathbf{U}}^{n+1},\label{702}\\
{\mathbf{U}}^{n+1} &=& {\mathbf{U}}^{n} + \delta t \dot{\mathbf{U}}^{n} + \left(\frac{1}{2}-\widetilde{\beta}\right)\delta t^ 2\,
\ddot{\mathbf{U}}^{n} + \widetilde{\beta}\delta t^2\,\ddot{\mathbf{U}}^{n+1},\label{703}
\end{eqnarray}
where $\widetilde{\beta}$ and $\widetilde{\gamma}$ are parameters of the methods that will be fixed later.
Replacing \eqref{702}--\eqref{703} in the equation of motion \eqref{LEM}, we obtain
\begin{equation}
      \left(
    \mathbf{M}
    +  \widetilde{\beta}\delta t^2\,
    \mathbf{K}
\right)
\ddot{\mathbf{U}}^{n+1}
+ 
	\mathbf{C} \halfscript{\overset{\alpha,\eta}{\mathbf{U}}}{n+1}
=
 -
  \mathbf{K}
\left( \mathbf{U}^n+\delta t \dot{\mathbf{U}}^n + \left(\dfrac{1}{2}- \widetilde{\beta}\right) \delta t^2 \ddot{\mathbf{U}}^n\right).
 \label{NewmarkW}
\end{equation}

\subsection{Approximation of fractional derivatives}
In order to numerically simulate the improper integral \eqref{623}, we consider $R>0$ sufficiently large, so that 
\begin{equation*}
\label{623b}
D^{\alpha,\eta}{\mathbf{U}}(t) \approx  2\mathfrak{C}\int_0^R\mu(\xi)
\mathbf{\Phi}(t,\,\xi)d\xi.
\end{equation*}
(we note the parity of the function $\mu\mathbf{\Phi}$ with respect to $\xi$ 
from \eqref{203} and \eqref{624}).
Let $\xi_\ell:=\ell\delta \xi$
 $\ell=1,\ldots,M$, 
 $\delta\xi=M/R$.
From \eqref{203}, we define
\begin{eqnarray*}
\label{629}\mu_\ell = |\xi_\ell|^{(2\alpha - 1)/2},\quad
\ell=1,\ldots,M,\quad 0 < \alpha < 1.
\end{eqnarray*}
Thus, an approximation of the integral \eqref{623}, is given by
\begin{eqnarray}
\label{631}
\mathbf{C}\halfscript{\overset{\alpha,\eta}{\mathbf{U}}}{n}\approx
2\mathfrak{C}  \displaystyle\sum_{\ell=1}^M\mu_\ell \mathbf{\Phi}^{n}_\ell.
\end{eqnarray}
On the other hand, the system \eqref{624}-\eqref{625} can be discretized using the Crank–Nicolson method \cite{Crank}, in order to maintain the conservation of energy, or its 
non-decrease in case of dissipation.
Combining it with the Newmark scheme \eqref{NewmarkW}
we obtain
the following  conservative numerical scheme:
\begin{equation}
    \begin{cases}
      \left(
    \mathbf{M}
    +  \widetilde{\beta}\delta t^2\,
    \mathbf{K}
\right)
\ddot{\mathbf{U}}^{n+1}
+ 
	\mathfrak{C}  \displaystyle\sum_{\ell=1}^M\mu_\ell \mathbf{\Phi}^{n+1}_\ell
=
 -
  \mathbf{K}
\left( \mathbf{U}^n+\delta t \dot{\mathbf{U}}^n + \left(\dfrac{1}{2}- \widetilde{\beta}\right) \delta t^2 \ddot{\mathbf{U}}^n\right),\\
\mathbf{\Phi}^{n+1}_\ell = \mathbf{\Phi}^{n}_\ell
- \delta t \left(\xi_\ell^2 +\eta\right) \mathbf{\Phi}^{n+\frac{1}{2}}_\ell
+ \delta t \mu_\ell  \dot{\mathbf{U}}^{n+\frac{1}{2}},
\end{cases}
 \label{ConservNumerical}
\end{equation}
where 
$\mathbf{\Phi}^{n+\frac{1}{2}}_\ell
=\dfrac{\mathbf{\Phi}^{n}_\ell+\mathbf{\Phi}^{n+1}_\ell}{2}
$
Using  \eqref{702} again and replacing in 
\eqref{ConservNumerical}$_2$,
we can rewrite \eqref{ConservNumerical}
in the following more explicit and computable way:
\begin{equation}
    \begin{cases}
      \left(
    \mathbf{M}
    +  \widetilde{\gamma}\delta t\,
    \mathbf{C}_{augm}
    +  \widetilde{\beta}\delta t^2\,
    \mathbf{K}
\right)
\ddot{\mathbf{U}}^{n+1}
=
-
	\mathfrak{C}  \displaystyle\sum_{\ell=1}^M\widetilde{\mu}_\ell \mathbf{\Phi}^{n}_\ell
 -\mathbf{C}_{augm}\left(
\dot{\mathbf{U}}^n + \left( 1 - \widetilde{\gamma} \right)
\delta t \ddot{\mathbf{U}}^n
\right)
  \\
\qquad\qquad\qquad\qquad\qquad\qquad\qquad - 
\mathbf{K}
\left( \mathbf{U}^n+\delta t \dot{\mathbf{U}}^n + \left(\dfrac{1}{2}- \widetilde{\beta}\right) \delta t^2 \ddot{\mathbf{U}}^n\right),
\\
\mathbf{\Phi}^{n+1}_\ell = \dfrac{2-\delta t\left(\xi_\ell^2 +\eta\right)}{2+\delta t\left(\xi_\ell^2 +\eta\right)} \mathbf{\Phi}^{n}_\ell
+ 
\dfrac{2\delta t \mu_\ell }{2+\delta t\left(\xi_\ell^2 +\eta\right)}
\dot{\mathbf{U}}^{n+\frac{1}{2}},
\end{cases}
 \label{ExplicitWay}
\end{equation}
where  $\widetilde{\mu}_\ell=\dfrac{2-\delta t\left(\xi_\ell^2 +\eta\right)}{2+\delta t\left(\xi_\ell^2 +\eta\right)}
\mu_\ell$ and $\mathbf{C}_{augm}=\delta t \mathfrak{C}  \left(\displaystyle\sum_{\ell=1}^M
\dfrac{2\widetilde{\mu}_\ell^2}{2+\delta t\left(\xi_\ell^2 +\eta\right)}\right)\textbf{I}_{2J+1} $.

\subsection{Decay of the discrete energy}

Evaluating \eqref{ConservNumerical}$_1$ in $t=t_{n+\frac{1}{2}}$, multiplying 
by $\ddot{\mathbf{U}}^{n+\frac{1}{2}}$, and summing
\eqref{ConservNumerical}$_2$ multiplied by 
$\mathfrak{C}\mathbf{\Phi}^{n+\frac{1}{2}}_\ell$, we obtain that
\begin{eqnarray}
\left[ E_\Delta \right]_n^{n+1} &:=&
 \left[ 
 \dfrac{1}{2} \dot{\mathbf{U}}^T  \mathbf{M} \dot{\mathbf{U}}
 +
  \dfrac{1}{2} {\mathbf{U}}^T  \mathbf{K} {\mathbf{U}}
  +
 \dfrac{\mathfrak{C}}{2}  \displaystyle\sum_{\ell=1}^M {\mu}_\ell \left|\mathbf{\Phi}_\ell\right|^2
  \right]_n^{n+1}\\
  &=&\nonumber
  - \ \mathfrak{C}
   \displaystyle\sum_{\ell=1}^M  \left(\xi_\ell^2 +\eta\right) \left|\mathbf{\Phi}^{n+\frac{1}{2}}_\ell\right|^2
   \label{ener_Mbodje}
\end{eqnarray}
which is consistent with the estimates \eqref{213}--\eqref{314}, and constitutes a correct approximation of the energy and its decreasing behavior.

\subsection{A first example: Reflection and refraction of a wavefront}

We start with a simple example of a wavefront interacting with the domain edges and the transmission interface. To do this, let us consider the initial condition:
\begin{equation}
    \label{CI}
\begin{cases}
\displaystyle u(x,0)=e^{-\frac{(x-x_0)^2}{\varepsilon}}, & u_t(x,0)=0,\qquad -L<x<0 \\
\displaystyle v(x,0)=0, & v_t(x,0)=0,\qquad 0<x<L,
\end{cases}
\end{equation}
with $L=1$, $x_0=-0.5$ and $\varepsilon=0.005$.
In this example we consider the physical parameters
$\rho_1=\rho_2=1$
and $k_1=10$, $k_2=2$.
Additionally, $\alpha=0.5$
and $\eta=1.0$
On the other hand, the discretization is given by a refinement of the space-time variables defined by
$N=500$ and $J=500$.
Additionally, we consider here
$R=10$ and 
$M=10000$.
\\
In Figure 1, the simulation of wave propagation with and without dissipation is observed. Beyond the comparison between the two phenomena, we observe that in both cases, the reflection phenomena at the domain edges, and reflection and refraction in the transmission condition occur properly, which validates the finite volume method for spatial discretization. On the other hand, temporal discretization using the Newmark method coupled with a Crank-Nicolson method ensures conservation and non-growth of discrete energy in each of the two cases respectively, as calculated in \eqref{ener_Mbodje}.

\begin{figure}[hbt]
\begin{center}
    \begin{subfigure}[t]{0.5\textwidth}
        \centering
        \includegraphics[scale=0.4]{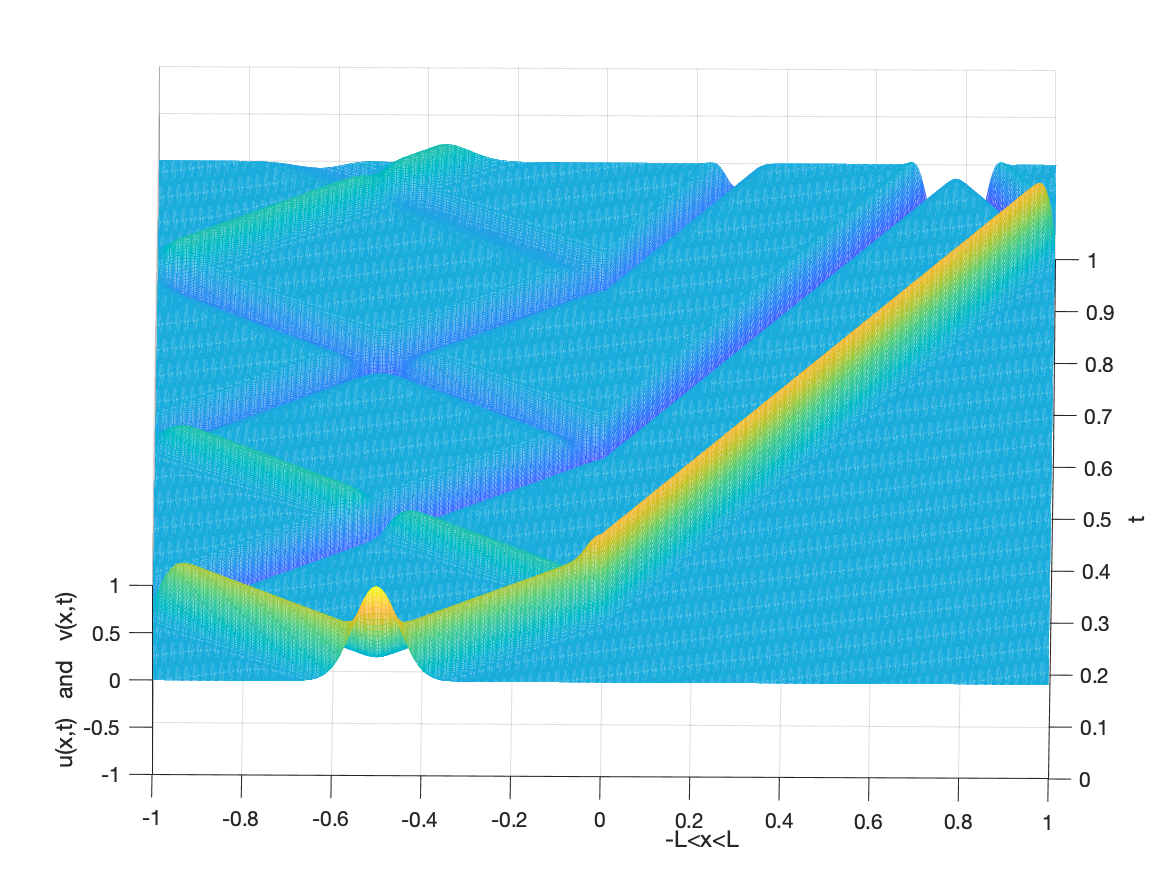}
        \caption{Wavefront without dissipation term.}
    \end{subfigure}%
    \begin{subfigure}[t]{0.5\textwidth}
        \centering
        \includegraphics[scale=0.4]{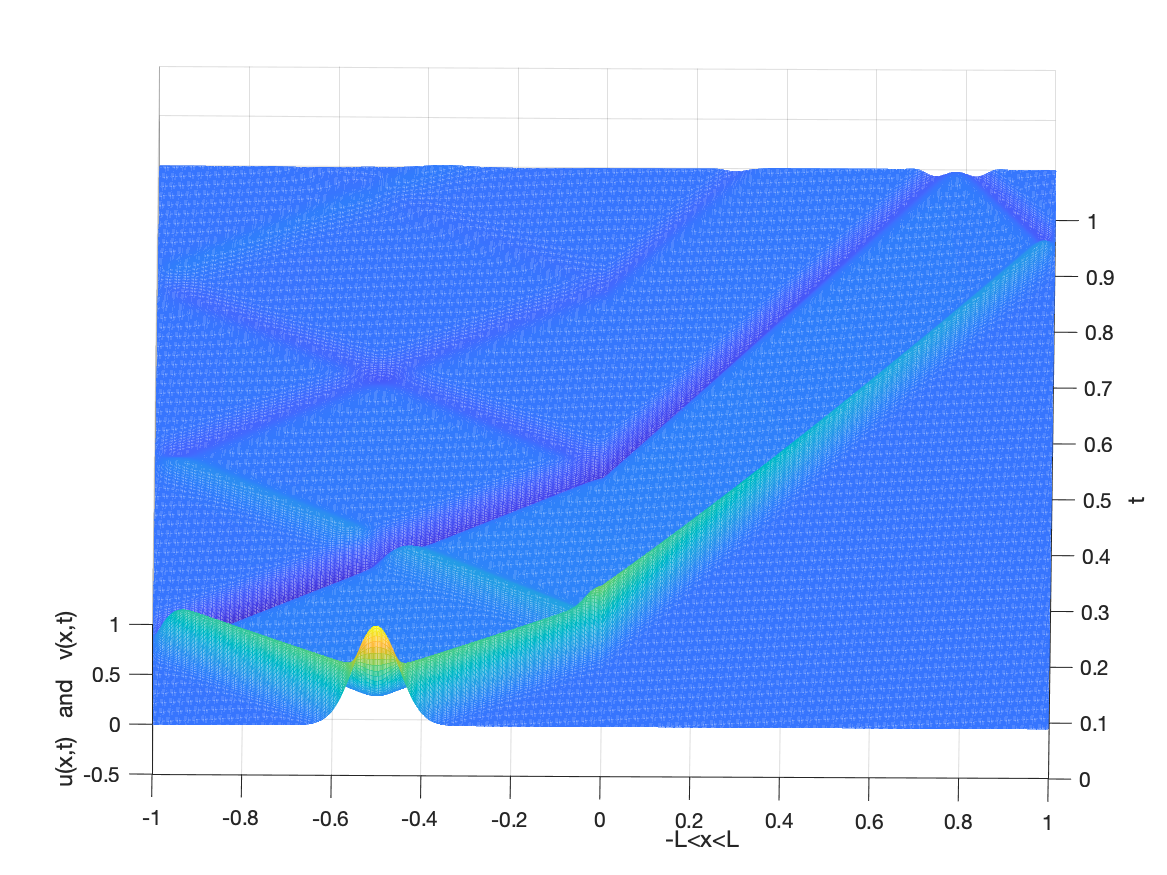} 
        \caption{Wavefront with fractional derivative term.}
    \end{subfigure}%
    \\
\caption{Simulation of the wavefront
for $0<t<1$, with and without dissipation terms.
}\label{Fig2}
\end{center}
\end{figure}
\subsection{Polynomial decay of energy}
Although the numerical decay of energy is visually evident, demonstrating that the asymptotic behavior is polynomial rather than exponential is not as straightforward. This distinction requires a sufficiently long time frame to become apparent. Furthermore, the numerical results are sensitive to both the choice of initial conditions and the values of the parameters $\alpha$ and $\eta$. For instance, while the parameter $\eta$ serves to regularize the equation’s solution, excessively high values of $\eta$ can destabilize our numerical scheme, highlighting the need for careful selection of both parameters and initial conditions to accurately capture the theoretical asymptotic behavior. For this purpose, we select a sufficiently large time interval $T = 10^4 \, \text{sec}$ and an initial condition in $D(A)$ given by
\begin{equation}
    \label{CIbis}
\begin{cases}
\displaystyle u(x,0)=\mathrm{sech}^2\left(-\frac{(x+x_0)^2}{\varepsilon}\right), & u_t(x,0)=0,\qquad -L<x<0 \\
\displaystyle v(x,0)=-\mathrm{sech}^2\left(-\frac{(x-x_0)^2}{\varepsilon}\right), & v_t(x,0)=0,\qquad 0<x<L,
\end{cases}
\end{equation}
with $L=1$, $x_0=0.5$ and $\varepsilon=0.005$.
As in the previous example, we consider the same physical parameters
$\rho_1=\rho_2=1$,
$k_1=10$, $k_2=2$,
and $\alpha=0.5$, 
and  the numerical parameters
$J=500$,
$R=10$ and 
$M=10000$.
On the other hand, we consider in this example, a discretization with $N=10^5$ time steps, and we perform simulations for values 
of $\eta$ equal to $0$, $0.0003$, $0.0003$,
$0.0005$, $0.0003$, $0.0007$, $0.001$
and $0.01$.
\\
In Figure 1, the asymptotic behavior of the energy is evident when $\eta=0$ different from when $\eta\not0$. In the first case ($\eta=0$), the asymptotic behavior is close to the curve $y=C_0 t^{-1}$, with $C_0=2.62$, drawn as a straight line on a log-log scale graph, in accordance with the statement of proposition \ref{prop5.2}. On the other hand, when $\eta>0$, we obtain in the same Figure 1, a package of curves of the Energy for different values of $\eta$, all of them bounded by a curve $y=C_4 t^{-4}$, with $C_4=1.75\times 10^{10}$, in accordance with the statement of proposition \ref{prop5.3}. Indeed, according to Proposition \ref{prop5.3}, if $\alpha=0.5$ it is obtained that the rate of decay of the energy is given by $-\dfrac{2}{1-\alpha}=-4$.
\begin{figure}
    \centering
    \includegraphics[width=0.5\linewidth]{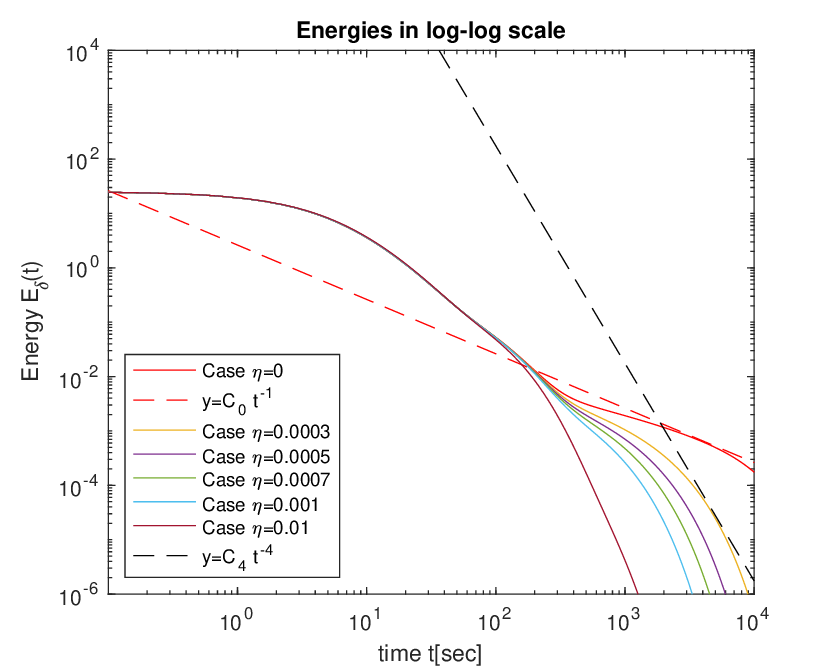}
    \caption{Polynomial decay of the energy for different values of $\eta$
    and comparison with the polynomial bounds $y=C_0t^{-1}$ and $y=C_4t^{-4}$;
    $C_0=2.62$, $C_4=1.75\times 10^{10}$.}
    \label{fig:trans_etas}
\end{figure}

\section*{Conflict of interest} 
There is no potential conflict of interest.

\bibliographystyle{elsarticle-num}

\begin{thebibliography}{00}

\bibitem{alves} M. Alves, J. E. Mu\~{n}oz Rivera, M. Sep\'{u}lveda, M. Zegarra Garay and O. Vera Villagr\'{a}n. {\it The asymptotic behaviour of the linear transmission problem in viscoelasticity.} Math. Nachr. 287. 5-6 (2014), 483--497.

\bibitem{kais} K. Ammari, H. Fathi and L. Robbiano. {\it Fractional-Feedback stabilization for a class of evolution systems.} J. Differential Equations. 268 (2020), 5751--5791.

\bibitem{kaisbis} K. Ammari, F. Hassine and L. Robbiano. {\it Stabilization for Some Fractional-Evolution Systems.} SpringerBriefs in Mathematics. Springer. Cham. 2022.

\bibitem{bardos} C. Bardos, G. Lebeau, and J. Rauch. {\it Sharp sufficient conditions for the observation, control, and stabilization of waves from the boundary.} SIAM J. Control Optim. 30. 5 (1992), 1024--1065.

\bibitem{batty} C. J. K. Batty, R. Chill and Y. Tomilov. {\it Fine scales of decay of operator semigroups.} J. Eur. Math. Soc. 18 (2016), 853--929.

\bibitem{Arendt}  W. Arendt and C. J. K. Batty. {\it Tauberian theorems and stability of one-parameter semigroups}. Trans. Amer. Math. Soc. 306. 2 (1988), 837--852.

\bibitem{bena} A. Benaissa and A. Boudaouad. {\it Some results on the energy decay of solutions for a wave equation with a general internal feedback of diffusive type.} Advances in Partial Differential Equations and Control. Trends Math. Birkh\"{a}user/Springer. Cham. 2024.

\bibitem{bo} A. Borichev and Y. Tomilov. {\it Optimal polynomial decay of function and operator semigroups.} Math. Ann. 347 (2010), 455--478

\bibitem{Brezis} H. Brezis. {\it Analise fonctionnelle: th\'{e}orie et applications.} Masson. Paris. 1983.

\bibitem{C1} M. Caputo. {\it Linear models of dissipation whose $Q$ is almost frequency independent  Part II.}  Geophys.  J. R. Astr.  Soc.

\bibitem{C2} M.  Caputo. {\it Elasticitá e Dissipazione}. Zanichelli,
Bologna (1969). (in Italian: Elasticity and Dissipation).

\bibitem{C3} M. Caputo,  F.  Mainardi.  {\it Linear models of dissipation in
anelastic solids.}  Riv.  Nuovo Cimento.  (Ser.  II) 1(1971), 161--198.

\bibitem{CM} J. Choi and R.  Maccamy.
{\it Fractional order Volterra equations with applications to
elasticity.} J. Math. Anal. Appl. 139 (1989), 448--464.

\bibitem{chen} G. Chen, S. A. Fulling, F. J. Narcowich and S. Sun, {\it Exponential decay of energy of evolution equations with locally distributed damping.} SIAM J. Appl. Math. 51. 1 (1991), 266--301.

\bibitem{Crank}
J. Crank and P. Nicolson, {\it A practical method for numerical evaluation of solutions of partial differential equations of the heat conduction type}. Proc. Camb. Phil. Soc. 43 (1),
 (1947), 50–-67

\bibitem{Eymard}
R. Eymard, T. Gallouët and R. Herbin,
{\it Finite volume methods.}
Handbook of Numerical Analysis,
Vol. 7, (2000), 713--1018.

\bibitem{ho} L.F. Ho. {\it Exact controllability of the one dimensional wave equation with locally distributed control.} SIAM J. Control Optim. 28. 3 (1990), 733--748.

\bibitem{huang} F.L. Huang.
{\it Characteristic condition for exponential stability of linear
dynamical systems in Hilbert spaces.} Ann. Diff. Eqns., Fuzhou.
1 (1985), 43--56.

\bibitem{KST} A. Kilbas, H. Srivastava and J. Trujillo.
{\it Theory and Applications of Fractional Differential Equations}.
North-Holland Mathematics Studies, 204. Elsevier Science B.V., Amsterdam. (2006). xvi+523.

\bibitem{liu1} K. Liu and Z. Liu. {\it Exponential decay of energy of the Euler-Bernoulli beam with locally distributed Kelvin-Voigt damping.} SIAM J. Control Optim. 36(3) (1998), 1086--1098.

\bibitem{liu} Z. Liu and S. Zheng. {\it Semigroups associated with Dissipative Systems.} CRC Reseach Notes in Mathematics 398. Chapman and Hall. (1999).

\bibitem{mar} P. Martinez. {\it Decay of solutions of the wave equation with a local highly degenerate dissipation.} Asymptotic Analysis. 19. 1 (1999), 1--17.

\bibitem{15} B. Mbodje.
{\it Wave energy decay under fractional derivative controls.} IMA J. Math. Contr. Inform. 23 (2006), 237--257.

\bibitem{nakao} M. Nakao. {\it Decay of solutions of the wave equation with a local nonlinear dissipation.} Math. Ann. 305. 3 (1996), 403--417.

\bibitem{Newmark}
N.M. Newmark.
{\it A method of computation for structural dynamics.}
J. Engrg. Mech. Div., ASCE. 85 (1959).

\bibitem{Pazy} A. Pazy.
{\it Semigroups of Linear Operators and Applications to Partial
Differential Equations.} Springer-Verlag. New York. (1983).

\bibitem{SKM} S. Samko, A. Kilbas and O. Marichev. {\it Integral and derivatives of fractional order.} Gordon Breach. New York. (1993).

\bibitem{zua} E. Zuazua. {\it Exponential decay for the semilinear wave equation with locally distributed damping.} Commun. Partial
Differ. Equations. 15. 2 (1990), 205--235.

\bibitem{zuazua} E. Zuazua. {\it Exponential decay for the semilinear wave equation with localized damping in unbounded domains.} J. Math.
Pures Appl. 70. 4 (1991), 513--529.

\end{thebibliography}

\end{document}